\documentclass[a4paper,10pt]{article}
\usepackage[top=1in, bottom=1in, left=1in, right=1in]{geometry}
\usepackage{amsfonts}
\usepackage{amssymb}
\usepackage{amsthm}
\usepackage{amsmath,amssymb,latexsym} 
\usepackage{fancyhdr}

\newtheorem{thm}{Theorem}

\newtheorem{prop}[thm]{Proposition}
\newtheorem{lem}[thm]{Lemma}
\newtheorem{cor}[thm]{Corollary}

\begin{document}

\title{Counting with 3-valued truth tables of bracketed formulae connected by implication }
\author{\textit{Volkan Yildiz} }
\date{vo1kan@hotmail.co.uk}
\maketitle

\abstract
In this paper we investigate the combinatorical structure of the Kleene type truth tables
of all bracketed formulae with n distinct variables connected by the binary connective
of implication.
\\\\\\


\pagebreak

\chead{ hello}
\section{Notations}
\begin{itemize}
\item $p_1,...,p_n$, and  $\phi, \; \psi$ are all distinct propositional variables.
\item `True' will be denoted by 1
\item `False' will be denoted by 0
\item `Unknown' will be denoted by 2
\item The set of counting numbers is denoted by $\mathbb{N}$
\item : such that
\item $\nu$ is the valuation function $\; : \;$ $\nu(\phi)=1\;$ if $\phi$ is true,  $\nu(\phi)=0\;$ if $\phi$ is false, and
 $\nu(\phi)=2\;$ if $\phi$ is unknown. 
\item $\wedge, \; \vee$ are the conjunction and disjunction operators.
\item $\Rightarrow$ the implication operator
\item $\neg$ the negation operator
\item $\#c$ denotes the case number in $t_n^{\#c}$
\item For the coefficient of $x^n$ in $G(x)$
$
[x^n]G(x)=[x^n]\bigg(\sum_{n\geq1} g_nx^n\bigg)=g_n
$
\end{itemize}

\section{Preface}
We have written in former papers about counting in truth tables in 2012. This paper was written in 2013, but never been published
due to time constraints and changes in my living conditions. Now during the pandemic period, I have some time in my hands to return to
my incomplete work. For 2-valued truth tables counting arguments one can refer back to my paper `General combinatorical structure of truth tables of bracketed formulae connected with implication', \cite{V1}.

\section{Intro}
Recall that the number of bracketings of a product of n terms is the Catalan number:
\[
C_n=\frac{1}{n}{{2n-2}\choose{n-1}}
\]
and its generating function 
\[
C(x)=\frac{1-\sqrt{1-4x}}{2}.
\]
Here we define the implication operator in the following way, ``Kleene's way"
\begin{displaymath}
\begin{array}{|c|c| c|c|}
\hline
\Rightarrow &1 & 0 & 2  \\ 
\hline 
 1 & 1 & 0 &2\\
\hline
0& 1 & 1&1\\
\hline
2 & 1 & 2 & 2\\
\hline

\end{array}
\end{displaymath}

\begin{prop}
Let $g_n$ be the total number of rows in all Kleene truth tables for bracketed implication with n distinct variables $p_1, . . . , p_n$. Then
\[
g_n=\sum_{i=1}^{n-1}g_ig_{n-i}, \;\; g_1=3
\]
\end{prop}
\begin{proof}
\begin{equation}\notag
g_n=3^nC_n=3^n\sum_{i=1}^{n-1}C_iC_{n-i}=\sum_{i=1}^{n-1}( 3^i C_i)( 3^{n-i}C_{n-i})=\sum_{i=1}^{n-i} g_ig_{n-i}
\end{equation}
\end{proof}
Thus it has the following generating function.
\begin{equation}
G(x)=(1-\sqrt{1-12x})/2
\end{equation}

\begin{prop} Let $f_n,\; t_n$ and $\;u_n$ be the number of rows with the value “false” , "true", and "unknow" in the Kleene truth tables of all
bracketed formulae with n distinct propositions $p_1, . . . , p_n$ connected by the binary connective
of implication. Then
$u_n$ has the following reccurence relation, and generating function $U(x)$.

\[
u_n=\sum_{i=1}^{n-1} u_ig_{n-i} =\sum_{i=1}^{n-1} u_i3^{n-i}C_{n-i}, \;\; u_1=1
\]

\[U(x)= \frac{1-\sqrt{1-12x}}{6}
\]
\end{prop}
\begin{proof}
\[u_n = t_iu_{n-i}+u_{i}f_{n-i}+u_i u_{n-i} = t_iu_{n-i} + u_i(g_{n-i} - t_{n-i})\]
Summing over n, gives us
\[
U(x)=x+U(x)G(x)
\]
Solving this for $U(x)$ gives us the required result.
\end{proof}
\begin{cor}
The number of unknown entries in bracketed Kleene's truth table connected by the implication is given by 
\begin{equation}
u_n=\frac{3^{n-1}}{n} {{2n-2}\choose{n-1}}
\end{equation}
\end{cor}

\begin{prop}
$f_n$ has the following reccurence relation, 
\[
f_n=\sum_{i=1}^{n-1} f_i\bigg( 2C_{n-i}3^{n-i-1}-f_{n-i} \bigg),\;\; f_1=1
\]
and $f_n$ has the following generating function
\[ F(x) = \frac{-2-\sqrt{1-12x}+\sqrt{5+24x+4\sqrt{1-12x}}}{6}
\]
\end{prop}

\begin{proof}
Consider
\[
\begin{split}
f_n  &= f_it_{n-i}\\
&=  f_i(g_{n-i}-f_{n-i}-u_{n-i})\\
&=  f_i3^{n-i}C_{n-i}- f_if_{n-i} - f_i3^{n-i-1}C_{n-i} \\
&= f_i(2C_{n-i}3^{n-i-1}-f_{n-i})
\end{split}
\]
Summing over n, gives us
\[
F(x)=2F(x)U(x)-F(x)^2+x.
\]
Solving it for  $F(x)$ gives us the required result. 
\end{proof}

\begin{cor}
$t_n$ has the following generating function
\[
T(x)=\frac{4-\sqrt{1-12x}-\sqrt{5+24x+4\sqrt{1-12x}}}{6}
\]
\end{cor}

\section{Asymptotic 1}

\begin{thm}
Let $f_n$, \; $t_n$\; and $u_n$ be number of rows with the value false, true and unknown in the Kleene truth tables of all the bracketed
implications with n variables. Then we have the following asymptotics
\[
f_n\sim \bigg(\frac{7-2\sqrt{7}}{21}  \bigg)\frac{12^{n-1}3}{\sqrt{\pi n^3}}
\]
\\
\[
t_n\sim \bigg(\frac{7+2\sqrt{7}}{21}  \bigg)\frac{12^{n-1}3}{\sqrt{\pi n^3}}
\]
\\
\[
u_n\sim \bigg(\frac{1}{3}  \bigg)\frac{12^{n-1}3}{\sqrt{\pi n^3}}
\]
\\
\[
g_n \sim \frac{12^{n-1}3}{\sqrt{\pi n^3}}
\]
\end{thm}

\begin{proof}

Recall 
\[ F(x) = \frac{-2-\sqrt{1-12x}+\sqrt{5+24x+4\sqrt{1-12x}}}{6}\]
By using the asymptotic techniques that we have discussed in \cite{V1}, we have $r=\frac{1}{12}$ and $F(\frac{1}{12})\not=0$. So let 
$A(x)=F(x)-F(\frac{1}{12})$.
\[
\begin{split}
\lim_{x \to \frac{1}{12}} \frac{A(x)}{f(x)}  &= \lim_{x\to\frac{1}{12}}\frac{-\sqrt{1-12x}+\sqrt{5+24x+\sqrt{1-12x}}-\sqrt{7}}{6\sqrt{1-12x}}\\
&= \lim_{x\to \frac{1}{!2}} \frac{-\sqrt{5+24x+4\sqrt{1-12x}}+2\sqrt{1-12x}-2}{6\sqrt{5+24x+4\sqrt{1-12x}}}\\
&= \frac{7-2\sqrt{7}}{-42}
\end{split}
\]
Therefore
\[
f_n\sim  \frac{7-2\sqrt{7}}{-42} {{n-3/2}\choose{n}}\bigg(\frac{1}{12}\bigg)^{-n}\sim  \bigg(\frac{7-2\sqrt{7}}{21}\bigg) \frac{12^{n-1}3}{\sqrt{\pi n^3}}
\]
With similar arguments we have the above asymptotics for $t_n$ and $u_n$.
\end{proof}

\begin{cor}
The number of rows with unknown in the Kleene truth tables is the average of 
the number of rows with true and false.
\[
 \frac{f_n+t_n}{2}=u_n, \;\;\; \forall n\geq1.
\]
\end{cor}
\begin{proof}
Since $u_n=3^{n-1}C_n \; \forall n\geq 1$, we have
\begin{equation}\notag
\begin{split}
f_n + t_n & = \frac{2}{3}g_n\\
3(f_n+t_n)&= 2g_n\\
f_n+t_n &= 2g_n-2f_n-2t_n\\
f_n+t_n&=2u_n\\
u_n&=\frac{f_n+t_n}{2}.
\end{split}
\end{equation}

\end{proof}
$\;$\\
Note here 
\[ \frac{7-2\sqrt{7}}{21} \approx 0.0813570180, \;\;
\frac{7+2\sqrt{7}}{21} \approx 0.5853096486 ,\;\; \frac{1}{3}\approx 0.333333333\]
\\
\[
\frac{7-2\sqrt{7}}{21} +
\frac{7+2\sqrt{7}}{21}=\frac{2}{3}
\]
The below table shows the sequences which we have discussed so far, up to n = 9.
\begin{center}
 \begin{tabular}{||c c c c c c c c c c c ||} 
 \hline\hline
$n$ & 1 & 2 & 3 & 4 & 5 & 6 & 7 & 8 & 9 &  \\ [0.5ex] 
 \hline\hline
$t_n$ & 1 & 5 & 30 & 229 & 1938 & 17530& 165852 & 1621133 & 16242474 &  \\ 
 \hline
$f_n$ & 1 & 1 & 6 & 41 & 330 & 2882 & 26604 & 255313 & 2521986 &\\
 \hline
$u_n$ & 1 & 3 & 18 & 135 & 1134& 10206 & 96228& 938223 & 9382230&\\
 \hline
$g_n$ & 3 & 9  & 54 & 405 & 3402 & 30618 & 288684& 2814669 & 28146690& \\
 \hline
 \hline
\end{tabular}
\end{center}
\section{Generation of truth table sequences}

Since $g_n=t_n+f_n+u_n, \; \forall n\geq 1$, we have
\begin{equation}\notag
\begin{split}
g_n & =\sum g_ig_{n-i} \\
& = \sum (t_i+f_i+u_i)(t_{n-i}+f_{n-i}+u_{n-i})\\
& = \underbrace{\sum t_it_{n-i}}_{t_n^{\#1}} + \underbrace{\sum t_if_{n-i}}_{f_n} + \underbrace{\sum t_iu_{n-i}}_{u_n^{\#1}} +
\underbrace{\sum f_it_{n-i}}_{t_n^{\#2}} +  \underbrace{\sum f_if_{n-i}}_{t_n^{\#3}} \\
 &+ \underbrace{\sum f_iu_{n-i}}_{t_n^{\#4}} + \underbrace{\sum u_it_{n-i}}_{t_n^{\#5}} + \underbrace{\sum u_if_{n-i}}_{u_n^{\#2}} + \underbrace{\sum u_iu_{n-i}}_{u_n^{\#3}}
\end{split}
\end{equation}
We can generate nine more sequences: $t_n^{\#1},\; t_n^{\#2},\;t_n^{\#3},\;t_n^{\#4},\;t_n^{\#5},\;f_n,\;u_1^{\#1},\;u_n^{\#2},\;$ and $u_n^{\#3},\;$ except from $f_n$, all other sequences equals to 0 when $n=1$. Each of these sequences, (and their generating functions) counts different rows of the corresponding truth table.
E.g. 
Let $\phi$ and $\psi$ be propositional variables, then
\begin{equation}\notag
\begin{split}
\nu(\phi \Rightarrow \psi) = 1\; &:\; (\nu(\phi) = 1 = \nu(\psi))\\
\nu(\phi \Rightarrow \psi) = 1\; &:\; (\nu(\phi) = 0 \wedge \nu(\psi)=1)\\
\nu(\phi \Rightarrow \psi) = 1\; &:\; (\nu(\phi) = 0 =\nu(\psi))\\
\nu(\phi \Rightarrow \psi) = 1\; &:\; (\nu(\phi) = 0 \wedge  \nu(\psi)=2)\\
\nu(\phi \Rightarrow \psi) = 1\; &:\; (\nu(\phi) = 2 \wedge \nu(\psi)=1).
\end{split}
\end{equation}
In each case we are interested in formulae obtained from $p_1\Rightarrow...\Rightarrow p_n$ by inserting
brackets such that the valuation of the first $i$ bracketing and the rest $(n-i)$ bracketing both give 1, ‘true’;  
such that the valuation of the first $i$ bracketing is 0 and the rest $(n-i)$ bracketing gives 1; 
such that the valuation of the first $i$ bracketing and the rest $(n-i)$ bracketing both give 0; 
such that the valuation of the first $i$ bracketing is 0 and and the rest $(n-i)$ bracketing gives 2; 
such that the valuation of the first $i$ bracketing is 2 and and the rest $(n-i)$ bracketing gives 1, respectively.
To get the corresponding generating function for $t_n^{\#1}$ \; we can make the following calculations
\[
t_n^{\#1}=\sum_{n=1}^{n-1} t_it_{n-1}
\]
Summing over $n$ gives us $T_1(x)= T(x)^2$. 
Using the same method we can obtain the following generating functions: 
$T_2(x)= F(x)T(x),$ \; $T_3(x)=F(x)^2$, \;    $T_4(x)=F(x)U(x)$, \; and  $T_5(x)=U(x)T(x)$.\\\\
A few terms for these \textit{ fresh } sequences:
\begin{equation}\notag
\begin{split}
(t_n^{\#1})_{n>0}&= 0,1, 10, 85, 758, 7066, 68180, 675725, 6840190, 70431982, 735446924,...\\
(t_n^{\#2})_{n>0}&= 0, 1, 6, 41, 330, 2882, 26604, 255313, 2521986, 25473638, 261898548,...\\
(t_n^{\#3})_{n>0}&= 0, 1, 2, 13, 94, 778, 6916, 64613, 625478, 6219070, 63138652,...\\
(t_n^{\#4})_{n>0}&= 0, 1, 4, 27, 212, 1830, 16760, 159963, 1573732, 15846354, 162518600,...\\
(t_n^{\#5})_{n>0}&= 0, 1, 8, 63, 544, 4974, 47392, 465519, 4681088, 47952810, 498672736,...\\
(t_n)_{n>0}&= 1, 5, 30, 229, 1938, 17530, 165852, 1621133, 16242474, 165923854, 1721675460,...
\end{split}
\end{equation}
Note that $\forall n\geq 1$ $\; t_n=\sum_{i=1}^5t_n^{\#i}$. With similar arguments we can get the 
rest of the generating functions and their corresponding sequences. 
\begin{equation}\notag
\begin{split}
\nu(\phi \Rightarrow \psi) = 0\; &:\; (\nu(\phi) = 1 \; \wedge \nu(\psi)=0)\\
\nu(\phi \Rightarrow \psi) = 2\; &:\; (\nu(\phi) = 1 \; \wedge \nu(\psi)=2)\\
\nu(\phi \Rightarrow \psi) = 2\; &:\; (\nu(\phi) = 2 \; \wedge  \nu(\psi)=0)\\
\nu(\phi \Rightarrow \psi) = 2\; &:\; (\nu(\phi) = 2 \; = \; \nu(\psi))\\
\end{split}
\end{equation}
$F(x)$ and $U(x)$ have been studied in former chapters. Here we want to  get more sequences from 
the original sequence  $u_n$, i.e. we want to break $u_n$ into $u_1^{\#1}$, $u_1^{\#2}$, and $u_1^{\#3}$. 
Moreover  the following corresponding generating function, and their sequences exist: $U_1(x)=T(x)U(x)$, $U_2(x)=U(x)F(x)$, and 
$U_3(x)=U(x)^2$.  
\begin{equation}\notag
\begin{split}
(u_n^{\#1})_{n>0}&= 0, 1, 8, 63, 544, 4974, 47392, 465519, 4681088, 47952810, 498672736,...\\
(u_n^{\#2})_{n>0}&= 0, 1, 4, 27, 212, 1830, 16760, 159963, 1573732, 15846354, 16251860,...\\
(u_n^{\#3})_{n>0}&= 0, 1, 6, 45, 378, 3402, 32076, 312741, 3127410, 31899582, 330595668,...\\
(u_n)_{n>0}&= 1, 3, 18, 135, 1134, 10206, 96228, 938223, 9382230, 95698746, 9917870040,...
\end{split}
\end{equation}
Consequently, the following observation is essential.
\begin{cor}
\[
g_n= \sum_{i=1}^5t_n^{\#i} +\sum_{i=1}^{3}u_n^{\#i} +f_n \sum_{i=1}^3t_n^{\#i} +2(t_n^{\#4}+t_n^{\#5})+u_n^{\#3} +f_n 
\]
\end{cor}
\pagebreak
\section{Asymptotics 2}

In this part we will be exploring the asymptotics of the sequences which we have seen in the former chapter.

\begin{lem}
Consider the sequences that we have discussed in the former chapter:
$t_n^{\#1}$, \; $t_n^{\#2}$, \;$t_n^{\#3}$, \;$t_n^{\#4}$, \; $t_n^{\#5}$,  and $u_n^{\#3}$,  then we have the following asymptotics
\begin{equation}\notag
\begin{split}
& t_n^{\#1}\sim \bigg(\frac{14+\sqrt{7}}{63}  \bigg)\frac{12^{n-1}3}{\sqrt{\pi n^3}},\;\;\;
f_n=t_n^{\#2}\sim  \bigg(\frac{7-2\sqrt{7}}{21}\bigg) \frac{12^{n-1}3}{\sqrt{\pi n^3}},\\\\
& t_n^{\#3}\sim  \bigg(\frac{11\sqrt{7}-28}{63}\bigg) \frac{12^{n-1}3}{\sqrt{\pi n^3}},\;\;\;
u_n^{\#2}= t_n^{\#4}\sim  \bigg(\frac{5\sqrt{7}-7}{126}\bigg) \frac{12^{n-1}3}{\sqrt{\pi n^3}},\\\\
& u_n^{\#1}=t_n^{\#5}\sim  \bigg(\frac{35-5\sqrt{7}}{126}\bigg) \frac{12^{n-1}3}{\sqrt{\pi n^3}},\;\;\;
u_n^{\#3}\sim  \bigg(\frac{1}{9}\bigg) \frac{12^{n-1}3}{\sqrt{\pi n^3}}.
\end{split}
\end{equation}
\end{lem}

where 
\begin{equation}\notag
\begin{split}
&\bigg(\frac{14+\sqrt{7}}{63}\bigg)\approx 0.2642182748,\;\;  \bigg(\frac{14+\sqrt{7}}{42}\bigg)\approx 0.0813570180\\\\
&\bigg(\frac{11\sqrt{7}-28}{63}\bigg)\approx 0.0175121337,\;\;  \bigg(\frac{5\sqrt{7}-7}{126}\bigg)\approx 0.04943457584\\\\
&\bigg(\frac{35-5\sqrt{7}}{126}\bigg)\approx 0.1727876464,\;\; \bigg(\frac{1}{9}\bigg)\approx 0.1111111111.
\end{split}
\end{equation}
\begin{proof}
Proofs are similar to the proof of theorem 6. No need to repeat the same type of calculations.
\end{proof}
\subsection{comparing asymptotics with 2 valued truth tables}
\begin{equation}\notag
\begin{split}
t_n^{\#1}&:\;\; \frac{1}{2} \to \bigg(\frac{14+\sqrt{7}}{63}  \bigg)\;\;\; \textit{decreased} \approx 47.2\% \\\\
t_n^{\#2}=f_n&:\;\; \frac{3-\sqrt{3}}{6} \to \bigg(\frac{7-2\sqrt{7}}{21}\bigg)\;\;\; \textit{decreased} \approx 61.5\%  \\\\
t_n^{\#3}&:\;\; \frac{2\sqrt{3}-3}{6} \to \bigg(\frac{11\sqrt{7}-28}{63}\bigg) \;\;\; \textit{decreased} \approx 77.4\%
\end{split}
\end{equation}

\begin{verse}
Her sabah yeni bir g\"un do\~garken\\
Bir g\"un de eksilir \"om\"urden\\
Her \c{s}afak bir h\i rs\i z gibidir\\
Elinde bir fenerle gelen.\\
\"O Hayyam.
\end{verse}
\end{document}